\documentclass[12pt]{article}
\usepackage{amsmath,amsfonts,amsthm,amssymb, mathtools}
\usepackage{mathrsfs, graphicx,color,latexsym, tikz, calc,cite}
 \usepackage{indentfirst}
\usetikzlibrary{shadows}
\usetikzlibrary{patterns,arrows,decorations.pathreplacing}
\usepackage[colorlinks=true,
linkcolor=blue,citecolor=blue,
urlcolor=blue]{hyperref}

\newtheorem{theorem}{Theorem}
\newtheorem{lemma}{Lemma}

\newtheorem{corollary}{Corollary}

\title{\bf \Large Splitting fields  of mixed Cayley graphs over abelian groups}

\author{Xueyi Huang$^a$,\ \ Lu Lu$^{b,}$\footnote{Corresponding author.}\setcounter{footnote}{-1}\footnote{\emph{E-mail address:} huangxymath@163.com (X. Huang), lulumath@csu.edu.cn (L. Lu), katja.moenius@mathematik.uniwuerzburg.de (K. M\"{o}nius).},\ \ Katja M\"{o}nius$^c$\\[2mm]
\small $^a$School of Mathematics, East China University of Science and Technology,\\
\small Shanghai 200237, P. R. China\\
\small $^b$School of Mathematics and Statistics, Central South University,\\
\small Changsha, Hunan, 410083, P. R. China\\
\small $^c$Institute of Mathematics, W\"{u}rzburg University, Emil-Fischer-Str. 40, \\
\small 97074 W\"{u}rzburg, Germany}

\date{ }

\begin{document}

\maketitle

\begin{abstract}
The splitting field $\mathbb{SF}(\Gamma)$ of a mixed graph $\Gamma$ is the smallest field extension of $\mathbb{Q}$ which contains all eigenvalues of the Hermitian adjacency matrix of $\Gamma$. The extension degree $[\mathbb{SF}(\Gamma):\mathbb{Q}]$ is called the algebraic degree of $\Gamma$. In this paper, we determine the splitting fields and algebraic degrees of mixed Cayley graphs over abelian groups. This generalizes the main results of [K. M\"{o}nius, Splitting fields of spectra of circulant graphs, J. Algebra 594(15) (2022) 154--169] and  [M. Kadyan, B. Bhattacharjya, Integral mixed Cayley graphs over abelian groups, Electron. J. Combin. 28(4) (2021) \#P4.46].

\end{abstract}

\baselineskip=0.202in

\section{Introduction}

All graphs considered in this paper have neither multiple edges nor loops. Let $\Gamma= (V, E, \vec{E})$ be a mixed graph with vertex set $V$, undirected edge set $E$ and directed edge (or arc) set  $\vec{E}$.  Here an undirected edge can be  viewed as two arcs in the opposite direction. We denote an undirected edge joining two vertices $u$ and $v$ in  $\Gamma$ by $uv$ or $vu$, and an arc from $u$ to $v$ by $(u,v)$. We say that $\Gamma$ is \textit{undirected} if  $\vec{E}=\emptyset$. 
The \textit{Hermitian adjacency matrix}  of $\Gamma$,  introduced  by Liu and Li \cite{LL15}, and independently by Guo and Mohar \cite{GM17}, is defined as $H(\Gamma)=(h_{uv})_{u,v\in V}$, where
$$
h_{uv}=\left\{
\begin{array}{ll}
1, & \mbox{if $uv\in E$,}\\
\mathbf{i}, & \mbox{if $(u,v)\in \vec{E}$,}\\
-\mathbf{i}, & \mbox{if $(v,u)\in \vec{E}$,}\\
0,& \mbox{otherwise.}
\end{array}
\right.
$$
Here $\mathbf{i}=\sqrt{-1}$. The \textit{Hermitian polynomial} of $\Gamma$, denoted by $p_H(\Gamma)$, is defined as the characteristic polynomial  of $H(\Gamma)$. The roots of $p_H(\Gamma)$ are called the   \textit{Hermitian eigenvalues}  of $\Gamma$.  Since  $H(\Gamma)$ is a Hermitian matrix with entries $0$, $1$ or $\pm \mathbf{i}$,   the Hermitian polynomial $p_H(\Gamma)$ is a monic polynomial with integer coefficients, and  hence every Hermitian eigenvalue of $\Gamma$ is an algebraic integer contained in some algebraic extension of $\mathbb{Q}$. The \textit{splitting field} of $\Gamma$, denoted by $\mathbb{SF}(\Gamma)$, is defined as the splitting field of $p_H(\Gamma)$, that is, the smallest field extension of $\mathbb{Q}$ over which $p_H(\Gamma)$  decomposes into linear factors. The \textit{algebraic degree} of $\Gamma$, denoted by $\mathrm{deg}(\Gamma)$, is the  degree of the least algebraic extension of $\mathbb{Q}$  which contains all Hermitian eigenvalues of $\Gamma$, i.e., $[\mathbb{SF}(\Gamma):\mathbb{Q}]$.  In particular, we say that $\Gamma$ is \textit{integral} if all its Hermitian eigenvalues are integers, or equivalently, $\mathrm{deg}(\Gamma)=1$.

Let  $G$ be a group with identity $1$, and let  $S$ be a subset of $G\setminus\{1\}$. For convenience,  we partition $S$ into two parts $S_1$ and $S_2$ such that  $S_1=S_1^{-1}$ and $S_2\cap S_2^{-1}=\emptyset$.  The \textit{Cayley graph}   $\mathrm{Cay}(G,S)$ is defined as the graph with vertex set $G$, with an undirected edge connecting two vertices $g,h\in G$ if $hg^{-1}\in S_1$ (and correspondingly, $gh^{-1}\in S_1$), and with an arc from $g$ to $h$ if  $hg^{-1}\in S_2$.  Clearly, $\mathrm{Cay}(G,S)$ is a mixed graph.  

In 1974, Harary and Schwenk \cite{HS74} asked for the characterization of integral (undirected) graphs, i.e., graphs with   algebraic degree  $1$.  In the past half century, this problem has aroused a great deal of interest, but is still far from resolved. For more results on integral graphs, especially  integral Cayley graphs, we refer the reader to \cite{AV09,ABM14,AP12,BCRSS02,BC76,B08,CFH19,CFYZ21,DKMA13,KB21,KS10,HL21,LHH18,MSP18,S05,WL05}, and references therein. As a natural generalization of  Harary and Schwenk's problem from an algebraic or number-theoretical viewpoint, M\"{o}nius, Steuding and Stumpf \cite{MSP18} proposed to determine the algebraic degree of a graph.  In \cite{M20,M22}, M\"{o}nius determined the splitting fields and algebraic degrees of  undirected Cayley graphs over cyclic groups. 

In this paper, we determine  the  splitting fields and algebraic degrees of mixed Cayley graphs over abelian groups, which  extends  the main result of  \cite{M22}. Furthermore, we provide a characterization of  integral mixed Cayley graphs over abelian groups, which implies the main result of \cite{KB21}.

\section{Main results}

In this section, we will determine  the  splitting fields and algebraic degrees of mixed Cayley graphs over abelian groups.  To achieve this goal, we need some notations and lemmas.

For a positive integer $t$, let $\zeta_{t}=\exp(2\pi\mathbf{i}/t)$ denote the primitive $t$-th root of unity, $\mathbb{Z}_{t}$ denote the ring  (or additive group) of integers module $t$, and $\mathbb{Z}_{t}^*=\{k\in \mathbb{Z}_{t}\mid \mathrm{gcd}(k,t)=1\}$ denote the multiplicative group of units in the ring $\mathbb{Z}_{t}$.

 From now on, we assume that $G$ is an abelian group under addition. For convenience, we write $G$ in  the form  
\begin{equation}\label{equ::0}
G=\mathbb{Z}_{n_1}\otimes \mathbb{Z}_{n_2}\otimes\cdots\otimes\mathbb{Z}_{n_r},
\end{equation}
where $n_i$ is a prime power for $1\leq i\leq r$. Clearly, $|G|=n_1n_2\cdots n_r$. Note that each $g\in G$ can be expressed as $g=(g_1,g_2,\ldots,g_r)$, where $g_i\in \mathbb{Z}_{n_i}$ for $1\leq i\leq r$.  For $g\in G$, let $\chi_g$ be the vector in $\mathbb{C}^{|G|}$ defined by
$$\chi_g(x)=\prod_{i=1}^r\zeta_{n_i}^{g_ix_i}, ~~\mbox{for all $x\in G$}.$$
It is well known that $\chi_g$,  $g\in G$, are  all  irreducible characters of $G$. For any subset $X\subseteq G$,   let 
$$\chi_g(X)=\sum_{x\in X}\chi_g(x),$$ 
and let  $\delta_X$ be the characteristic vector of $X$ over $G$ where $\delta_X(g) = 1$ if $g \in X$ and $\delta_X(g) = 0$ if $g \in G\setminus X$.

Based on  Babai's   formula  for  calculating  the  eigenvalues of Cayley color graphs (cf. \cite{B79}),  Kadyan and Bhattacharjya \cite{KB21} observed that the Hermitian eigenvalues of  mixed Cayley graphs over abelian groups can be expressed as follows.

\begin{lemma}[\cite{B79,KB21}]\label{lem::1}
Let $G$ be an abelian group (under addition) of the form \eqref{equ::0}, and let $S=A\cup B$ be a subset of $G\setminus \{0\}$ such that  $A=-A$ and $B\cap (-B)=\emptyset$. Then the Hermitian eigenvalues of the mixed Cayley graph $\Gamma=\mathrm{Cay}(G,S)$ are 
$$\gamma_g=\lambda_g+\mu_g,~~ \mbox{for all}~g=(g_1,\ldots,g_r)\in G,$$ 
where
$$
\lambda_g=\chi_{g}(A)=\sum_{a\in A}\prod_{i=1}^r\zeta_{n_i}^{g_ia_i},~~\mu_g=\mathbf{i}\left(\chi_g(B)-\chi_g(-B)\right)=\mathbf{i}\sum_{b\in B}\left(\prod_{i=1}^r\zeta_{n_i}^{g_ib_i}-\prod_{i=1}^r\zeta_{n_i}^{-g_ib_i}\right).
$$
\end{lemma}

The \textit{exponent} of $G$, denoted by  $exp(G)$, is the least common multiple of the orders of all elements in $G$. For the sake of simplicity, let
$$
n=exp(G)=\mathrm{lcm}(n_1,n_2,\cdots n_r).
$$ 
By Lemma \ref{lem::1}, the splitting filed $\mathbb{SF}(\Gamma)$ of $\Gamma=\mathrm{Cay}(G,S)$ satisfies  the relation
 $\mathbb{Q}\subseteq \mathbb{SF}(\Gamma)\subseteq\mathbb{Q}(\zeta_n,\mathbf{i})\subseteq \mathbb{Q}(\zeta_{4n})$. Therefore, in order to determine $\mathbb{SF}(\Gamma)$, it suffices to consider these fields between $\mathbb{Q}$ and $\mathbb{Q}(\zeta_{4n})$. 
 
Suppose that $\mathbb{K}$ is an arbitrary field satisfying $\mathbb{Q}\subseteq \mathbb{K}\subseteq \mathbb{Q}(\zeta_{4n})$. Borrowing the idea of \cite[Theorem 8]{KB21}, we have the following result.
\begin{lemma}\label{lem::2}
With reference to the notations in Lemma \ref{lem::1}, we have $\gamma_g\in \mathbb{K}$ if and only if $\lambda_g,\mu_g\in \mathbb{K}$, for all $g\in G$.
\end{lemma}
\begin{proof}
Let $g\in G$. Since $A=-A$, by Lemma \ref{lem::1}, 
$$\lambda_{-g}=\sum_{a\in A}\prod_{i=1}^r\zeta_{n_i}^{-g_ia_i}=\sum_{a\in -A}\prod_{i=1}^r\zeta_{n_i}^{g_ia_i}=\sum_{a\in A}\prod_{i=1}^r\zeta_{n_i}^{g_ia_i}=\lambda_{g}.$$ 
Furthermore, 
$$
\mu_{-g}=\mathbf{i}\sum_{b\in B}\left(\prod_{i=1}^r\zeta_{n_i}^{-g_ib_i}-\prod_{i=1}^r\zeta_{n_i}^{g_ib_i}\right)=-\mathbf{i}\sum_{b\in B}\left(\prod_{i=1}^r\zeta_{n_i}^{g_ib_i}-\prod_{i=1}^r\zeta_{n_i}^{-g_ib_i}\right)=-\mu_g.
$$
Therefore, 
$$
\lambda_g=\frac{\gamma_g+\gamma_{-g}}{2}~\textrm{and}~\mu_g=\frac{\gamma_g-\gamma_{-g}}{2},
$$
and the results follows immediately.
\end{proof}

Now consider the Galois group $\mathrm{Gal}(\mathbb{Q}(\zeta_{4n})/\mathbb{Q})$, where  $n=exp(G)$. For any $\sigma\in\mathrm{Gal}(\mathbb{Q}(\zeta_{4n})/\mathbb{Q})$, let $k_\sigma$ be the element of $\mathbb{Z}_{4n}^*$ such that $\sigma(\zeta_{4n})=\zeta_{4n}^{k_\sigma}$. It is known that the mapping $\eta: \mathrm{Gal}(\mathbb{Q}(\zeta_{4n})/\mathbb{Q})\rightarrow\mathbb{Z}_{4n}^*$ defined by  
$\eta(\sigma)=k_\sigma$, for all  $\sigma\in\mathrm{Gal}(\mathbb{Q}(\zeta_{4n})/\mathbb{Q})$, is a group isomorphism. Let $\mathbb{Z}_{4n}^*$ act on $G=\mathbb{Z}_{n_1}\otimes \mathbb{Z}_{n_2}\otimes\cdots\otimes\mathbb{Z}_{n_r}$ by setting $kg=k(g_1,g_2,\ldots,g_r)=(kg_1,kg_2,\ldots,kg_r)$ for any $k\in \mathbb{Z}_{4n}^*$ and $g\in G$. Let $\sigma\in\mathrm{Gal}(\mathbb{Q}(\zeta_{4n})/\mathbb{Q})$. For any $g_i\in \mathbb{Z}_{n_i}$, we see that
\begin{equation}\label{equ::1}
\sigma(\zeta_{n_i}^{g_i})=\sigma(\zeta_{4n}^{4ng_i/n_i})=\zeta_{4n}^{\eta(\sigma)\cdot 4ng_i/n_i}=\zeta_{n_i}^{\eta(\sigma)g_i}.
 \end{equation}
 Also note that
 \begin{equation}\label{equ::2}
\sigma(\mathbf{i})=\sigma(\zeta_4)=\sigma(\zeta_{4n}^n)=\zeta_{4n}^{\eta(\sigma)\cdot n}=\zeta_{4}^{\eta(\sigma)}=\mathbf{i}^{\eta(\sigma)}. 
 \end{equation}
 According to  \eqref{equ::1}, \eqref{equ::2} and Lemma \ref{lem::1},   for any $\sigma\in\mathrm{Gal}(\mathbb{Q}(\zeta_{4n})/\mathbb{Q})$ and  $g\in G$, we have
\begin{equation}\label{equ::3}
\begin{aligned}
\sigma(\lambda_g)&=\sigma(\chi_g(A))=\sigma\left(\sum_{a\in A}\prod_{i=1}^r\zeta_{n_i}^{g_ia_i}\right)=\sum_{a\in A}\prod_{i=1}^r\sigma(\zeta_{n_i}^{g_ia_i})\\
&=\sum_{a\in A}\prod_{i=1}^r\zeta_{n_i}^{\eta(\sigma)g_ia_i}=\chi_g(\eta(\sigma)A)
\end{aligned}
\end{equation}
and 
\begin{equation}\label{equ::4}
\begin{aligned}
\sigma(\mu_g)&=\sigma\left(\mathbf{i}\left(\chi_g(B)-\chi_g(-B)\right)\right)=\sigma\left(\mathbf{i}\left(\sum_{b\in B}\prod_{i=1}^r\zeta_{n_i}^{g_ibi}-\sum_{b\in B}\prod_{i=1}^r\zeta_{n_i}^{-g_ib_i}\right)\right)\\
&=\sigma(\mathbf{i})\left(\sum_{b\in B}\prod_{i=1}^r\sigma(\zeta_{n_i}^{g_ibi})-\prod_{i=1}^r\sigma(\zeta_{n_i}^{-g_ib_i})\right)\\
&=\mathbf{i}^{\eta(\sigma)}\left(\sum_{b\in B}\prod_{i=1}^r\zeta_{n_i}^{\eta(\sigma)g_ibi}-\prod_{i=1}^r\zeta_{n_i}^{-\eta(\sigma)g_ib_i}\right)\\
&=\mathbf{i}^{\eta(\sigma)}\left(\chi_g(\eta(\sigma)B)-\chi_g(-\eta(\sigma)B)\right),
\end{aligned}
\end{equation}
where $\eta(\sigma)A=\{\eta(\sigma)a\mid a\in A\}$ and $\eta(\sigma)B=\{\eta(\sigma)b\mid b\in B\}$. 

Let $F=\{k\in\mathbb{Z}_{4n}^*\mid k\equiv 1\pmod 4\}$. Clearly, $\mathbb{Z}_{4n}^*=F\cup (-F)$. We denote 
\begin{equation}\label{equ::5}
H=\eta(\mathrm{Gal}(\mathbb{Q}(\zeta_{4n})/\mathbb{K})),~~H_1=H\cap F~~\mbox{and}~~H_2=H\cap (-F).
\end{equation}
Then $H$ is a subgroup of $\mathbb{Z}_{4n}^*$ because $\mathrm{Gal}(\mathbb{Q}(\zeta_{4n})/\mathbb{K})$ is a subgroup of $\mathrm{Gal}(\mathbb{Q}(\zeta_{4n})/\mathbb{Q})$ and $\eta$ is an isomorphism from $\mathrm{Gal}(\mathbb{Q}(\zeta_{4n})/\mathbb{Q})$ to $\mathbb{Z}_{4n}^*$. 

The following two lemmas are crucial for the proof of our main result.

\begin{lemma}\label{lem::3}
With reference to the notations in Lemma \ref{lem::1}, we have $\lambda_g\in \mathbb{K}$ for all $g\in G$ if and only if $hA=A$ for all $h\in H$, where $H$ is given in  \eqref{equ::5}.
\end{lemma}
\begin{proof}
First we prove the sufficiency. For any  $\sigma\in\mathrm{Gal}(\mathbb{Q}(\zeta_{4n})/\mathbb{K})\subseteq \mathrm{Gal}(\mathbb{Q}(\zeta_{4n})/\mathbb{Q})$, we have $\eta(\sigma)\in H$, and  $\eta(\sigma)A=A$. Let $g\in G$. Then it follows from Lemma \ref{lem::1} and \eqref{equ::3} that
$$\sigma(\lambda_g)=\chi_g(\eta(\sigma)A)=\chi_g(A)=\lambda_g.$$
By the arbitrariness of $\sigma\in\mathrm{Gal}(\mathbb{Q}(\zeta_{4n})/\mathbb{K})$ and the fact that $\lambda_g\in \mathbb{Q}(\zeta_{4n})$, we must have $\lambda_g\in \mathbb{K}$, and the result follows.

Now consider the necessity. Let $h\in H$. Then there exists some $\sigma\in \mathrm{Gal}(\mathbb{Q}(\zeta_{4n})/\mathbb{K})\subseteq\mathrm{Gal}(\mathbb{Q}(\zeta_{4n})/\mathbb{Q})$ such that $\eta(\sigma)=h$. Since $\lambda_g\in \mathbb{K}$ for all $g\in G$, again by Lemma \ref{lem::1} and \eqref{equ::3}, we obtain 
\begin{equation}\label{equ::6}
\chi_g(hA)=\chi_g(\eta(\sigma)A)=\sigma(\lambda_g)=\lambda_g=\chi_g(A),~~\mbox{for all $g\in G$}.
\end{equation}
 Let $M$ denote the $|G|\times |G|$ matrix whose  $(x,y)$-entry is  $\chi_{x}(y)$ for $x,y\in G$. Then one can verify  that \eqref{equ::6} is actually equivalent to 
 $$M\delta_{hA}=M\delta_{A},$$ 
 where  $\delta_{hA}$ and $\delta_{A}$ are the characteristic vectors of $hA$ and $A$ over $G$, respectively. Note that $M$ is invertible by the orthogonal relations of irreducible characters of $G$ (see, for example, \cite{S77}). Therefore, we conclude that $\delta_{hA}=\delta_A$, and hence $hA=A$. By the arbitrariness of $h\in H$, the result follows.
\end{proof}

\begin{lemma}\label{lem::4}
With reference to the notations in Lemma \ref{lem::1}, we have $\mu_g\in \mathbb{K}$ for all $g\in G$ if and only if $h_1B=B$ for all  $h_1\in H_1$ and $h_2B=-B$ for all $h_2\in H_2$, where $H_1$ and $H_2$ are given in \eqref{equ::5}.
\end{lemma}
\begin{proof}
First we  prove the sufficiency. Let $g\in G$ and $\sigma\in \mathrm{Gal}(\mathbb{Q}(\zeta_{4n})/\mathbb{K})\subseteq \mathrm{Gal}(\mathbb{Q}(\zeta_{4n})/\mathbb{Q})$. If $\eta(\sigma)\in H_1$, then  $\eta(\sigma)\in F=\{k\in\mathbb{Z}_{4n}^*\mid k\equiv 1\pmod 4\}$ and $\eta(\sigma) B=B$. By Lemma \ref{lem::1} and \eqref{equ::4}, we have
$$
\sigma(\mu_g)=\mathbf{i}^{\eta(\sigma)}\left(\chi_g(\eta(\sigma)B)-\chi_g(-\eta(\sigma)B)\right)=\mathbf{i}\left(\chi_g(B)-\chi_g(-B)\right)=\mu_g.
$$
If $\eta(\sigma)\in H_2$, then $\eta(\sigma)\in -F=\{k\in\mathbb{Z}_{4n}^*\mid k\equiv 3\pmod 4\}$, $\eta(\sigma) B=-B$,  and 
$$
\sigma(\mu_g)=\mathbf{i}^{\eta(\sigma)}\left(\chi_g(\eta(\sigma)B)-\chi_g(-\eta(\sigma)B)\right)=-\mathbf{i}\left(\chi_g(-B)-\chi_g(B)\right)=\mu_g.
$$
Thus $\sigma(\mu_g)=\mu_g$ for all $\sigma \in \mathrm{Gal}(\mathbb{Q}(\zeta_{4n})/\mathbb{K})$, and it follows that $\mu_g\in \mathbb{K}$. By the arbitrariness of $g\in G$, the result follows.

Now consider the necessity. Let $h_1\in H_1$. Then there exists some $\sigma\in \mathrm{Gal}(\mathbb{Q}(\zeta_{4n})/\mathbb{K})\subseteq \mathrm{Gal}(\mathbb{Q}(\zeta_{4n})/\mathbb{Q})$ such that $h_1=\eta(\sigma)$. Since $\mu_g\in \mathbb{K}$ for all $g\in G$, again by Lemma \ref{lem::1} and \eqref{equ::4}, we obtain
$$
\mathbf{i}^{h_1}(\chi_g(h_1B)-\chi_g(-h_1B))=\sigma(\mu_g)=\mu_g=\mathbf{i}\left(\chi_g(B)-\chi_g(-B)\right),~~\mbox{for all $g\in G$}.
$$
As $\mathbf{i}^{h_1}=\mathbf{i}$ due to $h_1\in F=\{k\in\mathbb{Z}_{4n}^*\mid k\equiv 1\pmod 4\}$, the above condition becomes
\begin{equation}\label{equ::7}
\chi_g(h_1B)-\chi_g(-h_1B)=\chi_g(B)-\chi_g(-B),~~\mbox{for all $g\in G$}.
\end{equation}
 Let $M$ denote the $|G|\times |G|$ matrix whose  $(x,y)$-entry is  $\chi_{x}(y)$ for $x,y\in G$. Then  \eqref{equ::7} can be written as $$M(\delta_{h_1B}-\delta_{-h_1B})=M(\delta_B-\delta_{-B}),$$ 
where $\delta_{h_1B}$, $\delta_{-h_1B}$, $\delta_B$ and $\delta_{-B}$ are the characteristic vectors of $h_1B$, $-h_1B$, $B$ and $-B$ over $G$, respectively. Similarly as in the proof of Lemma \ref{lem::3}, $M$ is invertible, and hence
$$
\delta_{h_1B}-\delta_{-h_1B}=\delta_B-\delta_{-B}.
$$
Since $|B|=|h_1B|$,  $B\cap (-B)=\emptyset$ and $h_1B\cap (-h_1B)=\emptyset$, the above equality implies that $h_1B=B$. Therefore, $h_1B=B$ for all $h_1\in H_1$. Analogically, one can prove that $h_2B=-B$ for all $h_2\in H_2$.
\end{proof}

Recall that $F=\{k\in\mathbb{Z}_{4n}^*\mid k\equiv 1\pmod 4\}$ and $\mathbb{Z}_{4n}^*=F\cup (-F)$. We denote
\begin{equation}\label{equ::8}
\mathcal{H}=\{h\in\mathbb{Z}_{4n}^*\mid hA=A;~hB=B~\mbox{if}~h\in F~\mbox{and}~hB=-B~\mbox{if}~h\in -F\}.
\end{equation}
Clearly,  $\mathcal{H}$ is a subgroup of $\mathbb{Z}_{4n}^*$, and hence $\eta^{-1}(\mathcal{H})$ is a subgroup of $\mathrm{Gal}(\mathbb{Q}(\zeta_{4n})/\mathbb{Q})$. 

With the help of  Lemmas \ref{lem::1}--\ref{lem::4}, we now prove the main theorem of this paper.

\begin{theorem}\label{thm::main}
Let $G$ be an abelian group (under addition) with  $exp(G)=n$, and let  $S=A\cup B$ be  a subset of $G\setminus \{0\}$ such that  $A=-A$ and $B\cap (-B)=\emptyset$.  Then the splitting field of the mixed Cayley graph $\Gamma=\mathrm{Cay}(G,S)$ is 
$$
\mathbb{SF}(\Gamma)=\mathbb{Q}(\zeta_{4n})^{\eta^{-1}(\mathcal{H})}:=\{x\in\mathbb{Q}(\zeta_{4n})\mid \sigma(x)=x~\mbox{for all}~\sigma\in\eta^{-1}(\mathcal{H})\},
$$
where $\mathcal{H}$ is given in \eqref{equ::8} and $\eta$ is the isomorphism from $\mathrm{Gal}(\mathbb{Q}(\zeta_{4n})/\mathbb{Q})$ to $\mathbb{Z}_{4n}^*$ defined  above. Furthermore, the algebraic degree of $\Gamma$ is $$\mathrm{deg}(\Gamma)=\varphi(4n)/|\mathcal{H}|,$$ where $\varphi(\cdot)$ is the Euler's totient function.
\end{theorem}
\begin{proof}
Let $\mathbb{K}=\mathbb{Q}(\zeta_{4n})^{\eta^{-1}(\mathcal{H})}$. Clearly, $\eta^{-1}(\mathcal{H})=\mathrm{Gal}(\mathbb{Q}(\zeta_{4n})/\mathbb{K})$, or equivalently, $\mathcal{H}=\eta(\mathrm{Gal}(\mathbb{Q}(\zeta_{4n})/\mathbb{K}))$. According to the definition of $\mathcal{H}$, we have $hA=A$ for all $h\in\mathcal{H}$. By Lemma \ref{lem::3},  $\lambda_g\in \mathbb{K}$ for all $g\in G$. Also, $h_1B=B$  for all $h_1\in\mathcal{H}_1=\mathcal{H}\cap F$ and  $h_2B=-B$ for all $h_2\in\mathcal{H}_2=\mathcal{H}\cap (-F)$, and it follows from  Lemma \ref{lem::4}  that  $\mu_g\in \mathbb{K}$ for all $g\in G$. Hence, by Lemma \ref{lem::2}, $\gamma_g\in \mathbb{K}$ for all $g\in G$. This implies that $\mathbb{K}$ contains all Hermitian eigenvalues of $\Gamma$, and so $\mathbb{SF}(\Gamma)\subseteq \mathbb{K}$.  

On the other hand, since  $\mathbb{SF}(\Gamma)$ contains all Hermitian eigenvalues of $\Gamma$,  by Lemma \ref{lem::2},  $\lambda_g,\mu_g\in \mathbb{SF}(\Gamma)$ for all $g\in G$. Let $H=\eta(\mathrm{Gal}(\mathbb{Q}(\zeta_{4n})/\mathbb{SF}(\Gamma)))$. Clearly, $H$ is a subgroup of $\mathbb{Z}_{4n}^*$ because $\mathrm{Gal}(\mathbb{Q}(\zeta_{4n})/\mathbb{SF}(\Gamma))$ is a subgroup of $\mathrm{Gal}(\mathbb{Q}(\zeta_{4n})/\mathbb{Q})$. For any $h\in H$, according to Lemmas \ref{lem::3} and \ref{lem::4}, we have $hA=A$, and $hB=B$ if $h\in F$, $hB=-B$ if $h\in -F$. Thus, $H$ is a subgroup of $\mathcal{H}$, or equivalently, $\eta^{-1}(H)=\mathrm{Gal}(\mathbb{Q}(\zeta_{4n})/\mathbb{SF}(\Gamma))$  is a subgroup of $\eta^{-1}(\mathcal{H})=\mathrm{Gal}(\mathbb{Q}(\zeta_{4n})/\mathbb{K})$. Hence, $\mathbb{K}\subseteq \mathbb{SF}(\Gamma)$. 

Therefore, we may conclude that $\mathbb{SF}(\Gamma)=\mathbb{K}=\mathbb{Q}(\zeta_{4n})^{\eta^{-1}(\mathcal{H})}$. Moreover, the algebraic degree of $\Gamma$ is 
$$
\mathrm{deg}(\Gamma)=[\mathbb{Q}(\zeta_{4n})^{\eta^{-1}(\mathcal{H})}:\mathbb{Q}]=\frac{[\mathbb{Q}(\zeta_{4n}): \mathbb{Q}]}{[\mathbb{Q}(\zeta_{4n}):\mathbb{Q}(\zeta_{4n})^{\eta^{-1}(\mathcal{H})}]}=\frac{|\mathbb{Z}_{4n}^*|}{|\eta^{-1}(\mathcal{H})|}=\frac{\varphi(4n)}{|\mathcal{H}|},
$$
where $\varphi(\cdot)$ is the Euler's totient function.
\end{proof}

In the remaining part of this section, we will present some applications of Theorem \ref{thm::main}. Before going further, we need the following  powerful lemma, whose proof comes from \cite{Q17}.
\begin{lemma}[\cite{Q17}]\label{lem::5}
Let $n$ and $m$ be  positive integers. For every $h\in\mathbb{Z}_{mn}^*$, there exists some $h'\in \mathbb{Z}_n^*$ such that $h'\equiv h\pmod n$. Conversely, for every $h'\in\mathbb{Z}_n^*$, there exists some $h\in\mathbb{Z}_{mn}^*$ such that $h\equiv h'\pmod {n}$. 
\end{lemma}
\begin{proof}
The former statement of the lemma is obvious. We only need to consider  the later one. Suppose  $h'\in \mathbb{Z}_n^*$. For $i\in\{0,1,\ldots,m-1\}$, let $h_i=h'+in$. Clearly, $h_i\equiv h'\pmod n$.  It suffices to show that there exists some $i_0\in\{0,1,\ldots,m-1\}$ such that $h_{i_0}\in\mathbb{Z}_{mn}^*$, i.e., $\mathrm{gcd}(h_{i_0},mn)=1$. Let $d$ be the largest positive divisor of $mn$ such that $\mathrm{gcd}(d,n)=1$. If $d=1$, then $mn$ and $n$ share the same prime divisors, and we can take $i_0=0$. Now suppose $d>1$.  Since  $\mathrm{gcd}(d,n)=1$, we have $d\mid m$, and hence $d\leq m$. We claim that $h_i\not\equiv h_j\pmod d$ for  $0\leq i<j\leq d-1$. In fact, if $h_i\equiv h_j\pmod d$, then $d\mid (j-i)n$, and hence $d\mid (j-i)$, contrary to $0<j-i<d$. Since $h_0,h_1,\ldots,h_{d-1}$ are distinct (module $d$), there exists some $i_0\in \{0,1,\ldots,d-1\}\subseteq \{0,1,\ldots,m-1\}$ such that $h_{i_0}\equiv 1\pmod d$, i.e., $\mathrm{gcd}(h_{i_0},d)=1$. Also note that $\mathrm{gcd}(h_{i_0},n)=\mathrm{gcd}(h',n)=1$. Therefore, $\mathrm{gcd}(h_{i_0},mn)=1$, as desired.
\end{proof}

Under the assumption of Theorem \ref{thm::main}, if  $B$ is empty, then $S=A=-A=-S$, and  $\Gamma=\mathrm{Cay}(G,S)$ would be  an undirected Cayley graph. By Lemma \ref{lem::1}, all  eigenvalues of  $\Gamma$ are contained in $\mathbb{Q}(\zeta_n)$. As above, let $\eta'$ denote the group isomorphsim from $\mathrm{Gal}(\mathbb{Q}(\zeta_n)/\mathbb{Q})$ to $\mathbb{Z}_n^*$ such that $\sigma(\zeta_n)=\zeta_n^{\eta'(\sigma)}$ for all $\sigma\in\mathrm{Gal}(\mathbb{Q}(\zeta_n)/\mathbb{Q})$. 
Note  that $\mathbb{Z}_n^*$ also acts on $G=\mathbb{Z}_{n_1}\otimes \mathbb{Z}_{n_2}\otimes\cdots\otimes\mathbb{Z}_{n_r}$ by setting $kg=k(g_1,\ldots,g_r)=(kg_1,\ldots,kg_r)$ for any $k\in\mathbb{Z}_n^*$ and $g\in G$. 

As an application of  Theorem \ref{thm::main}, we give the splitting fields and algebraic degrees of undirected Cayley graphs over abelian groups.

\begin{theorem}\label{thm::2}
Let $G$ be an abelian group (under addition) with  $exp(G)=n$, and let $S$ be a subset of $G\setminus\{0\}$ such that $S=-S$. Suppose that $\mathcal{H}=\{h\in\mathbb{Z}_{4n}^*\mid hS=S\}$ and $\mathcal{H}'=\{h'\in\mathbb{Z}_n^*\mid h'S=S\}$. Then  the splitting field of $\Gamma=\mathrm{Cay}(G,S)$ is
$$
\begin{aligned}
\mathbb{SF}(\Gamma)&=\mathbb{Q}(\zeta_{4n})^{\eta^{-1}(\mathcal{H})}:=\{x\in\mathbb{Q}(\zeta_{4n})\mid \sigma(x)=x~\mbox{for all}~\sigma\in\eta^{-1}(\mathcal{H})\}\\
&=\mathbb{Q}(\zeta_{n})^{{\eta'}^{-1}(\mathcal{H}')}:=\{x\in\mathbb{Q}(\zeta_{n})\mid \sigma'(x)=x~\mbox{for all}~\sigma'\in\eta^{-1}(\mathcal{H'})\},
\end{aligned}
$$
where $\eta$ and $\eta'$ are defined above.  Furthermore, the algebraic degree of $\Gamma$ is 
$$
\mathrm{deg}(\Gamma)=\varphi(4n)/|\mathcal{H}|=\varphi(n)/|\mathcal{H}'|,
$$
where $\varphi(\cdot)$ is the Euler's totient function.
\end{theorem}

\begin{proof}
By Theorem \ref{thm::main}, it suffices to prove that 
\begin{equation}\label{equ::9}
\mathrm{Gal}(\mathbb{Q}(\zeta_{4n})/\mathbb{Q}(\zeta_n)^{{\eta'}^{-1}(\mathcal{H}')})=\eta^{-1}(\mathcal{H}).
\end{equation}
Let $\sigma\in\eta^{-1}(\mathcal{H})$. Then  $\sigma(\zeta_{n})=\sigma(\zeta_{4n}^4)=\zeta_{4n}^{4\eta(\sigma)}=\zeta_{n}^{\eta(\sigma)}$. Since $\eta(\sigma)\in\mathbb{Z}_{4n}^*$, by Lemma \ref{lem::5}, there exists some $k_\sigma'\in\mathbb{Z}_n^*$ such that  $\eta(\sigma)\equiv k_\sigma'\pmod n$, and hence $\zeta_n^{\eta(\sigma)}=\zeta_n^{k_\sigma'}$. This implies that  $\sigma|_{\mathbb{Q}(\zeta_n)}=\sigma':={\eta'}^{-1}(k_\sigma')$. Also note that  $k_\sigma'S=\eta(\sigma)S=S$ because $\eta(\sigma)\equiv k_\sigma'\pmod n$. Then  $\sigma(x)=\sigma'(x)=x$ for all $x\in\mathbb{Q}(\zeta_n)^{{\eta'}^{-1}(\mathcal{H}')}$, and hence $\sigma\in \mathrm{Gal}(\mathbb{Q}(\zeta_{4n})/\mathbb{Q}(\zeta_n)^{{\eta'}^{-1}(\mathcal{H}')})$. Hence, by the arbitrariness of  $\sigma\in\eta^{-1}(\mathcal{H})$, we obtain
\begin{equation}\label{equ::10}
\eta^{-1}(\mathcal{H})\le\mathrm{Gal}(\mathbb{Q}(\zeta_{4n})/\mathbb{Q}(\zeta_n)^{{\eta'}^{-1}(\mathcal{H}')}).
\end{equation}
Let  $\mathcal{P}(G)$ denote the power set of $G$, i.e.,  the collection of all subsets of $G$. Let $\mathbb{Z}_{4n}^*$ and $\mathbb{Z}_{n}^*$ act on $G$ as mentioned above. Then they  act on $\mathcal{P}(G)$ naturally.  By Lemma \ref{lem::5},  it is easy to see that the actions of $\mathbb{Z}_{4n}^*$ and $\mathbb{Z}_n^*$ on $\mathcal{P}(G)$ share the same orbits. Let $\mathcal{O}_S$ denote the common orbit of $\mathbb{Z}_{4n}^*$ and $\mathbb{Z}_{n}^*$ containing $S\in \mathcal{P}(G)$. Note that  $\mathcal{H}$ and $\mathcal{H}'$ are actually the stabilizer subgroups of $\mathbb{Z}_{4n}^*$ and $\mathbb{Z}_{n}^*$ with respect to $S$, respectively. Therefore, by the orbit-stabilizer theorem,
\begin{equation}\label{equ::11}
\frac{\varphi(4n)}{|\mathcal{H}|}=\frac{|\mathbb{Z}_{4n}^*|}{|\mathcal{H}|}=|\mathcal{O}_S|=\frac{|\mathbb{Z}_{n}^*|}{|\mathcal{H}'|}=\frac{\varphi(n)}{|\mathcal{H}'|}.
\end{equation}
Then 
\begin{equation}\label{equ::12}
\begin{aligned}
|\mathrm{Gal}(\mathbb{Q}(\zeta_{4n})/\mathbb{Q}(\zeta_n)^{{\eta'}^{-1}(\mathcal{H}')})|&=\left[\mathbb{Q}(\zeta_{4n}):\mathbb{Q}(\zeta_n)^{{\eta'}^{-1}(\mathcal{H}')}\right]\\
&=\left[\mathbb{Q}(\zeta_{4n}):\mathbb{Q}(\zeta_n)\right]\cdot \left[\mathbb{Q}(\zeta_{n}):\mathbb{Q}(\zeta_n)^{{\eta'}^{-1}(\mathcal{H}')}\right]\\
&=\frac{\varphi(4n)}{\varphi(n)}|\mathcal{H}'|\\
&=|\mathcal{H}|,
\end{aligned}
\end{equation}
where the last equality follows from \eqref{equ::11}.  Combining \eqref{equ::10} and \eqref{equ::12}, we  obtain  \eqref{equ::9} immediately. Furthermore, by Theorem \ref{thm::main} and \eqref{equ::11}, the algebraic degree of $\Gamma$ is 
$\mathrm{deg}(\Gamma)=\varphi(4n)/|\mathcal{H}|=\varphi(n)/|\mathcal{H}'|$, as desired.
\end{proof}

By Theorem \ref{thm::2}, we obtain the following result due to M\"{o}nius \cite{M22} immediately.
\begin{corollary}[\cite{M22}]\label{cor::1}
Let $S$ be a subset of $\mathbb{Z}_n\setminus\{0\}$ such that $S=-S$, and let $\mathcal{H}'=\{h'\in\mathbb{Z}_n^*\mid h'S=S\}$. Then  the splitting field of $\Gamma=\mathrm{Cay}(\mathbb{Z}_n,S)$ is
$$
\mathbb{SF}(\Gamma)=\mathbb{Q}(\zeta_{n})^{{\eta'}^{-1}(\mathcal{H}')}:=\{x\in\mathbb{Q}(\zeta_{n})\mid \sigma'(x)=x~\mbox{for all}~\sigma'\in\eta^{-1}(\mathcal{H'})\},
$$
where  $\eta'$ is defined above.  Furthermore, the algebraic degree of $\Gamma$ is 
$$
\mathrm{deg}(\Gamma)=\varphi(n)/|\mathcal{H}'|,
$$
where $\varphi(\cdot)$ is the Euler's  totient function.
\end{corollary}

As another application of Theorem \ref{thm::main}, we provide a characterization for integral mixed Cayley graphs over abelian groups.

\begin{theorem}\label{thm::3}
Let $G$ be an abelian group (under addition) with $exp(G)=n$, and let  $S=A\cup B$ be  a subset of $G\setminus \{0\}$ such that  $A=-A$ and $B\cap (-B)=\emptyset$.  Let $F=\{k\in\mathbb{Z}_{4n}^*\mid k\equiv 1\pmod 4\}$. Then the  mixed Cayley graph $\Gamma=\mathrm{Cay}(G,S)$ is integral if and only if  $kA=A$ for all $k\in \mathbb{Z}_{4n}^*$ and $k_1B=B$ for all $k_1\in F$, or equivalently, $A$ is a union of some orbits of $\mathbb{Z}_{4n}^*$ acting on $G$ and  $B$ is a union of some orbits of $F$ acting on $G$.
\end{theorem}
\begin{proof}

If $kA=A$ for all $k\in \mathbb{Z}_{4n}^*$ and $k_1B=B$ for all $k_1\in F$, then  $\mathcal{H}=\mathbb{Z}_{4n}^*$ by \eqref{equ::8}, and it follows from Theorem \ref{thm::main} that  $\mathrm{deg}(\Gamma)=1$. Conversely, if $\mathrm{deg}(\Gamma)=1$, again by Theorem \ref{thm::main}, we have $\mathcal{H}=\mathbb{Z}_{4n}^*$, and the result follows.
\end{proof}

Suppose that $G$ is an abelian group (under addition)  with $exp(G)=n$. For any $g\in G$, we denote by $o(g)$ the \textit{order} of $g$. Let $\mathcal{F}(G)$ be the family of all subgroups of $G$. The \textit{Boolean algebra} of $G$, denoted by $\mathcal{B}(G)$, is the set whose elements are obtained by arbitrary finite intersections, unions, and complements of the elements in $\mathcal{F}(G)$. The minimal non-empty elements of $\mathcal{B}(G)$ are called \textit{atoms}. Thus every element of $\mathcal{B}(G)$ is a union of some atoms. For each $g\in G$, the atom of $\mathcal{B}(G)$ containing $g$ is of the form (cf. \cite{AP12}) 
$$
[g]=\{g'\in G\mid \langle g'\rangle=\langle g\rangle\}=\{kg\mid k\in \mathbb{Z}_{o(g)}^*\}.
$$
Since $o(g)\mid n$, by Lemma \ref{lem::5}, the atom $[g]$ can also be  expressed as 
$$
[g]=\{kg\mid k\in \mathbb{Z}_{4n}^*\}.
$$
Therefore, the condition $kA=A$ for all $k\in \mathbb{Z}_{4n}^*$ in Theorem \ref{thm::3} is actually equivalent to 
$$
A\in \mathcal{B}(G).
$$
It remains to analyze the condition $k_1B=B$ for all $k_1\in F$ in Theorem \ref{thm::3}. 

If $4\nmid n$, we claim that the condition is equivalent to $B=\emptyset$. In fact, if $n\equiv 1\pmod 4$, then $2n-1\in F$, and hence  $(2n-1)B=B$. On the other hand, since the order of any element in $B$ is a divisor of $n$,  we have $(2n-1)B=-B$. Then $B=-B$, and hence $B=\emptyset$ due to $B\cap (-B)=\emptyset$.  Similarly, we can prove that $B=\emptyset$ by taking  $n-1\in F$ when $n\equiv 2\pmod 4$, and $2n-1\in F$ when $n\equiv 3\pmod 4$. 

Now suppose $4\mid n$. Let $G(4)=\{g\in G\mid o(g)\equiv 0\pmod 4\}$. We assert that the condition implies that $B\subseteq G(4)$. If not,  there exists some $x\in B$ such that $4\nmid o(x)$. Let $\ell$ be the largest positive integer such that $2^\ell\mid n$. Clearly, $\ell\geq 2$. Since $n/2^{\ell-1}$ and $n$ share the same prime divisors, and $n/2^{\ell-1}-1\equiv 1\pmod 4$,  we have  $n/2^{\ell-1}-1\in F$. Hence, $(n/2^{\ell-1}-1)x\in B$. On the other hand,   $(n/2^{\ell-1}-1)x=-x\in -B$ because $o(x)\mid n$ and $4\nmid o(x)$. Therefore, $B\cap (-B)\neq \emptyset$, a contradiction. For each $g\in G(4)$, we denote  
$$
[\![g]\!]=\{kg\mid k\in \mathbb{Z}_{o(g)}^*,~k\equiv 1\pmod 4\}.
$$
Since $o(g)\mid 4n$ and $4\mid n$, according to  the proof of Lemma \ref{lem::5}, we have
$$
[\![g]\!]=\{kg\mid k\in \mathbb{Z}_{4n}^*,~k\equiv 1 \pmod 4\}=\{kg\mid k\in F\}.
$$
Let $\mathcal{D}(G)$ denote the set of all subsets $D$ of $G$ such that $D$ is the union of some $[\![g]\!]$'s with $g\in G(4)$ (cf. \cite{KB21}).
Then we see that the condition $k_1B=B$ for all $k_1\in F$  is actually equivalent to 
$$
B\in \mathcal{D}(G).
$$

Concluding the above discussion, we obtain the following  result due to Kadyan and  Bhattacharjya \cite{KB21}.

\begin{corollary}[\cite{KB21}]\label{cor::2}
Let $G$ be an abelian group (under addition) with $exp(G)=n$, and let  $S=A\cup B$ be  a subset of $G\setminus \{0\}$ such that  $A=-A$ and $B\cap (-B)=\emptyset$. Then the  mixed Cayley graph $\Gamma=\mathrm{Cay}(G,S)$ is integral if and only if  $A\in \mathcal{B}(G)$, and $B=\emptyset$ when $4\nmid n$, $B\in\mathcal{D}(G)$ when $4\mid n$.
\end{corollary}

\section{Concluding remarks}
In this paper, we determine the splitting fields and algebraic degrees of mixed Cayley graphs over abelian groups. As applications, we obtain the splitting fields and algebraic degrees of undirected Cayley graphs over abelian groups, and provide a characterization for integral mixed Cayley graphs over abelian groups.

In \cite{MO20}, Mohar introduced the  \textit{Hermitian adjacency matrix of second kind} $HS(\Gamma)$ for a mixed graph  $\Gamma$, which is obtained from $H(\Gamma)$ by replacing  $\mathbf{i}$ with $e^{2\pi \mathbf{i}/6}=(1+\sqrt{3}\mathbf{i})/2$. Let $\Gamma$ be a mixed Cayley graph over an abelian group. By using the method developed in the present paper,  we can also determine the splitting field of the characteristic polynomial of $HS(\Gamma)$. Moreover, we can give a sufficient and necessary condition for  $HS(\Gamma)$ containing only integer eigenvalues.

\section*{Acknowledgements}
L. Lu is supported by National Natural Science Foundation of China (Grant No. 12001544) and Natural Science Foundation of Hunan Province  (Grant No. 2021JJ40707). X. Huang is supported by National Natural Science Foundation of China (Grant No. 11901540).

\end{document}